\documentclass{birkjour}
\usepackage{lmodern}
\usepackage{graphicx}

 \newtheorem{thm}{Theorem}[section]

 \theoremstyle{definition}
 
 \theoremstyle{remark}

 \numberwithin{equation}{section}

\begin{document}

\title[On bicomplex Pell and Pell-Lucas numbers]{On bicomplex Pell and Pell-Lucas numbers}

\author[F\"{u}gen Torunbalc{\i} Ayd{\i}n]{F\"{u}gen Torunbalc{\i}  Ayd{\i}n}
\address{%
Yildiz Technical University\\
Faculty of Chemical and Metallurgical Engineering\\
Department of Mathematical Engineering\\
Davutpasa Campus, 34220\\
Esenler, Istanbul,  TURKEY}

\email{ftorunay@gmail.com ; faydin@yildiz.edu.tr}

\thanks{*Corresponding Author}

   \begingroup
    \renewcommand{\thefootnote}{}
    \footnotetext{%
	    } 
    \endgroup


\begin{abstract}

In this paper, bicomplex Pell and bicomplex Pell-Lucas numbers are defined. Also, negabicomplex Pell and negabicomplex Pell-Lucas numbers are given. Some algebraic properties of bicomplex Pell and  bicomplex Pell-Lucas numbers which are connected with bicomplex numbers and Pell and Pell-Lucas numbers are investigated. Furthermore,  d'Ocagne's identity, Binet's formula, Cassini's identity and Catalan's identity for these numbers are given.
\\ \end{abstract} \

\maketitle

\section{Introduction}

\par   Bicomplex numbers were introduced by Corrado Segre in 1892 \cite{J}. G. Baley Price (1991), presented bicomplex numbers based on multi-complex spaces and functions in his book \cite{G}. In recent years, fractal structures of these numbers have also been studied \cite{H} . The set of bicomplex numbers can be expressed by the basis $\{1\,,i\,,j\,,i\,j\,\}$ as, 
\begin{equation}\label{E1}
\begin{aligned}
\mathbb{C}_2=\{\, q=q_1+iq_2+jq_3+ijq_4 \ | \ q_1,q_2,q_3,q_4\in \mathbb R\}
\end{aligned}
\end{equation}
or
\begin{equation}\label{E2}
\begin{aligned}
\mathbb{C}_2=\{\, q=(q_1+iq_2)+j(q_3+iq_4) \ | \ q_1,q_2,q_3,q_4\in \mathbb R\}
\end{aligned}
\end{equation} \\
where $i$,$j$ and $ij$ satisfy the conditions 
\begin{equation*}
i^2=-1,\,\,\,j^2=-1,\,\,\,i\,j=j\,i.  
\end{equation*} \,
\par Thus, any bicomplex number $q$ is introduced as pairs of typical complex numbers with the additional structure of commutative multiplication (Table 1).
\begin{table}[ht]
\caption{Multiplication scheme of bicomplex numbers} 
\centering 
\begin{tabular}{c c c c c} 
\\
\hline 
x & 1 & i & j & i\,j \\ [0.5ex] 
\hline 
1 & 1 & i & j & i\,j \\ 
i & i & -1 & i\,j & -j \\
j & j & i\,j & -1 & -i \\
i\,j & i\,j & -j & -i & 1 \\
 [1ex] 
\hline 
\end{tabular}
\label{table:nonlin} 
\end{table}
\par A set of bicomplex numbers $\mathbb{C}_2$ is a real vector space with addition and scalar multiplication operations. The vector space $\mathbb{C}_2$ equipped with bicomplex product is a real associative algebra. Also, the vector space together with the properties of multiplication and the product of the bicomplex numbers is a commutative algebra. Furthermore, three different conjugations can operate on bicomplex numbers \cite{D},\cite{H}-\cite{I} as follows: 

\begin{equation}\label{E3}
\begin{aligned}
q=q_1+i\,q_2+j\,q_3+i\,j\,q_4=(q_1+iq_2)+j\,(q_3+iq_4),\,\, q\in{\mathbb{C}_{2}}\\ 
{q_i}^*=q_1-iq_2+jq_3-ijq_4=(q_1-iq_2)+j\,(q_3-iq_4),\\
{q_j}^*=q_1+iq_2-jq_3-ijq_4=(q_1+iq_2)-j\,(q_3+iq_4),\\
{q_{ij}}^*=q_1-iq_2-jq_3+ijq_4=(q_1-iq_2)-j\,(q_3-iq_4).
\end{aligned}
\end{equation} \\
and properties of conjugation

\begin{equation}\label{E4}
\begin{aligned}
1)\,&(q^*)^*=q ,\,\,\,\,q_1,q_2\in{\mathbb{C}_{2}},\,\,\lambda,\mu\in{\mathbb{R}}\\ 
2)\,&({q_1}\,{q_2})^*={q_2}^*\,{q_1}^*,\\
3)\,&({q_1}+{q_2})^*={q_1}^*+{q_2}^*,\\
4)\,&(\lambda\,{q_1})^*=\lambda\,{q}^*,\\
5)\,&(\lambda\,{q_1}\pm\,\mu\,{q_2})^*=\lambda\,{q_1}^*+\mu\,{q_2}^*. \\ \\
\end{aligned}
\end{equation} \,\,
Therefore, the norm of the bicomplex numbers is defined as 

\begin{equation}\label{E5}
\begin{aligned}
{{N}_{q}}_{i}=\left\| {q\times{q}_{i}} \right\|=\sqrt[]{\left|{q}_{1}^2+{q}_{2}^2-{q}_{3}^2-{q}_{4}^2+2\,j\,({q}_{1}{q}_{3}+{q}_{2}{q}_{4})\right|}, \\ 
{{N}_{q}}_{j}=\left\| {q\times{q}_{j}} \right\|=\sqrt[]{\left|{q}_{1}^2-{q}_{2}^2+{q}_{3}^2-{q}_{4}^2+2\,i\,({q}_{1}{q}_{2}+{q}_{3}{q}_{4})\right|}, \\ 
{{N}_{q}}_{i\,j}=\left\| {q\times{q}_{i\,j}} \right\|=\sqrt[]{\left|{q}_{1}^2+{q}_{2}^2+{q}_{3}^2+{q}_{4}^2+2\,i\,j\,({q}_{1}{q}_{4}-{q}_{2}{q}_{3})\right|}.  
\end{aligned}
\end{equation} \\
\\
\par Pell numbers were invented by John Pell but, these numbers are named after Edouard Lucas. Pell and Pell-Lucas numbers have important parts in mathematics. They have fundamental importance in the fields of combinatorics and number theory \cite{A},\cite{C},\cite{F}. \\ 
\par The sequence of Pell numbers
$$\,\;1\,,\,\,2\,,\,\,5,\,\,12,\,\,29,\,\,70,\,\,169,\,\,408,\,\,985,\,\,2378, \ldots \,,{P_n}, \ldots $$
is defined by the recurrence relation
$${P_n} = 2\,{P_{n - 1}} + {P_{n - 2}}\;,\;\;(n \ge 3),$$
with ${P_{1}} =1 , {P_{2}} = 2\,$.\\  
\par The sequence of Pell - Lucas numbers
$$\,\;2\,,\,\,6\,,\,\,14,\,\,34,\,\,82,\,\,198,\,\,478,\,\,1154,\,\,2786,\,\,6726, \ldots \,,{Q_n}, \ldots $$
is defined by the recurrence relation
$${Q_{n}} = 2\,{Q_{n - 1}} + {Q_{n - 2}}\;,\;\;(n \ge 3),$$
with ${Q_{1}} =2 \,, {Q_{2}} = 6\,$.\\
Also, the sequence of modified Pell numbers
$$\,\;1\,,\,\,3\,,\,\,7,\,\,17,\,\,41,\,\,99,\,\,329,\,\,577,\,\,1393,\,\,3363, \ldots \,,{q_n}, \ldots $$
is defined by the recurrence relation
$${q_{n}} = 2\,{q_{n - 1}} + {q_{n - 2}}\;,\;\;(n \ge 3),$$
with ${q_{1}} =1 \,, {q_{2}} = 3\,$. \\

Furthermore, we can see the matrix representations of Pell and Pell-Lucas numbers in \cite{B},\cite{E}-\cite{K}. \\ \,
\\ \,\,
Also,  for Pell, Pell-Lucas and modified Pell numbers the following properties hold \cite{A},\cite{C},\cite{F},  
\begin{equation}\label {E6}
{P}_{m}{P}_{n + 1} +\,{P}_{m - 1}{P}_{n}={P}_{m+n} \,, \,\ 
\end{equation}
\begin{equation}\label {E7}
{P}_{m}{P}_{n + 1} - {P}_{m + 1}{P}_{n}=(-1)^n \,{P}_{m - n} \,, 
\end{equation}
\begin{equation}\label {E8}
{P}_{n - 1}{P}_{n + 1} - {{P}_{n}^2}=(-1)^{n} \,,
\end{equation}
\begin{equation}\label {E9}
{P}_{n}^2 +\,{P}_{n + 1}^2={P}_{2n + 1} \,,
\end{equation}
\begin{equation}\label {E10}
{P}_{n + 1}^2 -\,{P}_{n - 1}^2=2\,{P}_{2n} \,,
\end{equation}
\begin{equation}\label {E11}
2\,{P}_{n + 1}\,{P}_{n} - 2\,{P}_{n}^2={P}_{2n} \,, 
\end{equation}
\begin{equation}\label {E12}
{{P}_{n}^2} +\,{{P}_{n + 3}^2}=5\,{{P}_{2n + 3}} \,,
\end{equation}
\begin{equation}\label {E13}
{P}_{2n + 1} + {P}_{2n}=2\,{P}_{n + 1}^2- 2\,{P}_{n}^2- (-1)^n \,,
\end{equation}
\begin{equation}\label {E14}
{P}_{n}^2 +\,{P}_{n - 1}{P}_{n + 1}=\frac{Q_{n}^2}{4}  \,,
\end{equation}
\begin{equation}\label {E15}
{P}_{n + 1} + {P}_{n -1}={Q}_{n} \,,
\end{equation}
\begin{equation}\label {E16}
{P}_{n}\,{Q}_{n}={P}_{2n} \,,
\end{equation}
\begin{equation}\label {E17}
{Q}_{n}=2\,{q}_{n} \,,
\end{equation}
\begin{equation}\label {E18}
{P}_{n + 1} - {P}_{n}={q}_{n} \,,
\end{equation}
\begin{equation}\label {E19}
{P}_{n + 1} + {P}_{n}={q}_{n+1}.
\end{equation}
\\
\par In this paper, the bicomplex Pell and bicomplex Pell-Lucas numbers will be defined.
The aim of this work is to present in a unified manner a variety of algebraic properties of both the bicomplex numbers as well as the bicomplex Pell and Pell-Lucas numbers. In particular, using three types of conjugations, all the properties established for bicomplex numbers are also given for the bicomplex Pell and Pell-Lucas numbers. In addition, d'Ocagne's identity, the negabicomplex Pell and Pell-Lucas numbers, Binet's formula, Cassini's identity and Catalan's identity for these numbers are given.
\section{The bicomplex Pell and Pell-Lucas numbers} 

The bicomplex Pell and Pell-Lucas numbers $BP_{n}$ and  $BPL_{n}$ are defined by the basis $\{\,1\,,\,i\,,\,j\,,\,i\,j\,\}$ as follows
\begin{equation}\label{F1}
\begin{aligned}
\mathbb{C}^P_{2}=\{\,BP_{n}={P}_{n}
+i\,{P}_{n+1}+j\,{P}_{n+2}+i\,j\,{P}_{n+3}\,| \, {P}_{n},\,\,n-th\,Pell\, number\}.
\end{aligned}
\end{equation}
and
\begin{equation}\label{F2}
\begin{aligned}
{\mathbb{C}}^{PL}_{2}=\{\,BPL_{n}=&{Q}_{n}+i\,{Q}_{n+1}+j\,{Q}_{n+2}+i\,j\,{Q}_{n+3}\,|\, Q_n,\,\\ &\,n-th\,Pell-Lucas\,\, number\}
\end{aligned}
\end{equation}
where $i,\,j$ and $i\,j$ \,\,satisfy the conditions 
\begin{equation*}
i^2=-1,\,\,\,j^2=-1,\,\,\,i\,j=j\,i.  
\end{equation*} 
\par Starting from ${n=0}$\,, the bicomplex Pell and bicomplex Pell-Lucas numbers can be written respectively as;\\
\\
${BP}_{0} = 1\,i+2\,j+5\,i\,j\,,\,\;{BP}_{1} = 1+2\,i+5\,j+12\,i\,j\,,\,\,\,{BP}_{2} = 2+5\,i+12\,j+29\,i\,j\,,\dots$\\
\\
${BPL}_{0} = 2+2\,i+6\,j+14\,i\,j\,,\,\,\;{BPL}_{1} = 2+6\,i+14\,j+34\,i\,j\,,\\  {BPL}_{2} = 6+14\,i+34\,j+82\,i\,j\,,\dots$ \\

Let ${\,{BP}_{n}}$ and ${\,{BP}_{m}}$ be two bicomplex Pell  numbers such that
\begin{equation}\label{F3} 
{BP}_{n}={P}_{n}+i\,{P}_{n+1}+j\,{P}_{n+2}+i\,j\,{P}_{n+3}
\end{equation}
and 
\begin{equation}\label{F4} 
{BP}_{m}={P}_{m}+i\,{P}_{m+1}+j\,\,{P}_{m+2}+i\,j\,{P}_{m+3}.
\end{equation}
\\
Then, the addition and subtraction of these numbers are given by 
\begin{equation}\label{F5}
\begin{array}{rl}
{BP}_{n}\pm{BP}_{m}=&({P}_{n}+i\,{P}_{n+1}+j\,{P}_{n+2}+ij\,{P}_{n+3}) \\ &\pm ({P}_{m}+i\,{P}_{m+1}+j\,{P}_{m+2}+ij\,{P}_{m+3}) \\
=&({P}_{n}\pm{P}_{m})+i\,({P}_{n+1}\pm{P}_{m+1})+j\,({P}_{n+2}\pm{P}_{m+2}) \\ &+ ij\,({P}_{n+3}\pm{P}_{m+3}).
\end{array}
\end{equation}
The multiplication of a bicomplex Pell number by the real scalar ${\lambda}$ is defined as
\begin{equation}\label{F6}
\lambda{BP}_{n}=\lambda{P}_{n}+i\,\lambda{P}_{n+1}+j\,\lambda{P}_{n+2}+i\,j\,\lambda{P}_{n+3}.
\end{equation}
\\
The multiplication of two bicomplex Pell numbers is defined by 
\begin{equation}\label{F7}
\begin{array}{rl}
{\,{BP}_{n}}\times\,{\,{BP}_{m}}=&({P}_{n}+i\,{P}_{n+1}+j\,{P}_{n+2}+ij\,{P}_{n+3}) \\
&({P}_{m}+i\,{P}_{m+1}+j\,{P}_{m+2}+ij\,{P}_{m+3}) \\
=&({P}_{n}{P}_{m}-{P}_{n+1}{P}_{m+1}-{P}_{n+2}{P}_{m+2}-{P}_{n+3}{P}_{m+3})\\ & +i\,({P}_{n}{P}_{m+1}+{P}_{n+1}{P}_{m}-{P}_{n+2}{P}_{m+3}-{P}_{n+3}{P}_{m+2})\\ &+j\,({P}_{n}{P}_{m+2}+{P}_{n+2}{P}_{m}-{P}_{n+1}{P}_{m+3}-{P}_{n+3}{P}_{m+1}) \\ &+ i\,j\,({P}_{n}{P}_{m+3}+{P}_{n+3}{P}_{m}+{P}_{n+1}{P}_{m+2}+{P}_{n+2}{P}_{m+1}) \\ = &{\,{BP}_{m}}\times\,{\,{BP}_{n}}\,. \\
\end{array}
\end{equation}
Three kinds of conjugation can be defined for bicomplex numbers \cite{D}. Therefore, conjugation of the bicomplex Pell numbers is defined in three different ways as follows 
\begin{equation}\label{F8}
{\,({BP}_n)_i^*}={{P}_{n}}-i\,{{P}_{n+1}}+j\,{{P}_{n+2}}-i\,j\,{{P}_{n+3}},
\end{equation}
\begin{equation}\label{F9}
{\,({BP}_n)_j^*}={{P}_{n}}+i\,{{P}_{n+1}}-j\,{{P}_{n+2}}-i\,j\,{{P}_{n+3}},
\end{equation}
\begin{equation}\label{F10}
{\,({BP}_n)_{ij}^*}={{P}_{n}}-i\,{{P}_{n+1}}-j\,{{P}_{n+2}}+i\,j\,{{P}_{n+3}}.
\end{equation}\,
\begin{thm} 
Let ${{BP}_{n}}$ and ${{BP}_{m}}$  be two bicomplex Pell numbers. In this case, we can give the following relations between the conjugates of these numbers: \\
\begin{equation}\label{F11}
\begin{array}{rl}
({{BP}_{n}}\times{{BP}_{m}})_i^*=&({{BP}_{m}})_i^*\times({{BP}_{n})_i^*}=({{BP}_{n}})_i^*\times({{BP}_{m})_i^*},
\end{array}
\end{equation}
\begin{equation}\label{F12}
\begin{array}{rl}
({{BP}_{n}}\times{{BP}_{m}})_j^*=&({{BP}_{m}})_j^*\times({{BP}_{n})_j^*}=({{BP}_{n}})_j^*\times({{BP}_{m})_j^*},
\end{array}
\end{equation}
\begin{equation}\label{F13}
\begin{array}{rl}
({{BP}_{n}}\times{{BP}_{m}})_{i\,j}^*=&({{BP}_{m}})_{i\,j}^*\times({{BP}_{n})_{i\,j}^*}=({{BP}_{n}})_{i\,j}^*\times({{BP}_{m})_{i\,j}^*} \,.
\end{array}
\end{equation}
\begin{proof} It can be proved easily by using (2.8)-(2.10).
\end{proof}
\end{thm}
In the following theorem, some properties related to the conjugations of the bicomplex Fibonacci numbers are given.
 
\begin{thm} 
Let $({BP}_{n})_i^*$,\,$({BP}_{n})_j^*$ and $({BP}_{n})_{i\,j}^*$,  be three kinds of conjugation of the bicomplex Pell numbers. In this case, we can give the following relations:
\begin{equation}\label{F14}
\begin{array}{rl}
{\,{BP}_{n}}\times{\,({BP}_{n})_i^*}=&2(-\,Q_{2n+3}+j\,{P}_{2n+3}\,),
\end{array}
\end{equation}
\begin{equation}\label{F15}
\begin{array}{rl}
{\,{BP}_{n}}\times{\,({BP}_{n})_j^*}=& (P_{n}^2-\,{P}_{n+1}^2+P_{n+2}^2-\,{P}_{n+3}^2) \\ &+4\,i\,({P}_{2n+3}+{P}_{n}\,{P}_{n+1}),
\end{array}
\end{equation}
\begin{equation}\label{F16}
\begin{array}{rl}
{\,{BP}_{n}}\times{\,({BP}_{n})_{i\,j}^*}=&6\,{P}_{2n+3}+4\,i\,j\,(-1)^{n+1},
\end{array}
\end{equation}
\begin{equation}\label{F17}
\begin{array}{rl}
{\,{BP}_{n}}\times{\,({BP}_{n})_i^*}+{\,{BP}_{n-1}}\times{\,({BP}_{n-1})_i^*}=&2\,(\,8\,P_{2n+2}+\,j\,{Q}_{2n+2}),
\end{array}
\end{equation}
\begin{equation}\label{F18}
\begin{array}{rl}
{\,{BP}_{n}}\times{\,({BP}_{n})_j^*}+{\,{BP}_{n-1}}\times{\,({BP}_{n-1})_j^*}=&12\,(\,-P_{2n+2}+\,i\,{P}_{2n+2}),
\end{array}
\end{equation}
\begin{equation}\label{F19}
\begin{array}{rl}
{\,{BP}_{n}}\times{\,({BP}_{n})_{i\,j}^*}+{\,{BP}_{n-1}}\times{\,({BP}_{n-1})_{i\,j}^*}=&6\,Q_{2n+2}.
\end{array}
\end{equation}
\end{thm}
\begin{proof}
(2.14): By the equations (2.1) and (2.7) we get, \\
\begin{equation*}
\begin{array}{rl}
{BP}_{n}\times{({BP}_{n})_i^*}=&({P}_{n}^2+{P}_{n+1}^2-{P}_{n+2}^2-{P}_{n+3}^2) \\
&\quad \quad+2\,j\,({P}_{n}\,P_{n+2}+{P}_{n+1}\,{P}_{n+3}) \\
=&{P}_{2n+1}-{P}_{2n+5}+2\,j\,{P}_{2n+3} \\
=&2\,(-\,{Q}_{2n+3}+\,j\,{P}_{2n+3}\,). 
\end{array}
\end{equation*}
(2.15): By the equations (2.1) and (2.8) we get, \\
\begin{equation*}
\begin{array}{rl}
{\,{BP}_{n}}\times{\,({BP}_{n})_j^*}=&({{P}_{n}^2}-{{P}_{n+1}^2}+{{P}_{n+2}^2}-{P}_{n+3}^2) \\
&\quad \quad+2\,i\,({P}_{n}\,P_{n+1}+{P}_{n+2}\,P_{n+3}) \\
=&({{P}_{n}^2}-{{P}_{n+1}^2}+{{P}_{n+2}^2}-{P}_{n+3}^2) \\
&\quad \quad+4\,i\,(\,{P}_{2n+3}+{P}_{n}\,{P}_{n+1}). 
\end{array}
\end{equation*}
(2.16): By the equations (2.1) and (2.9) we get, \\
\begin{equation*}
\begin{array}{rl}
{\,{BP}_{n}}\times{\,({BP}_{n})_{i\,j}^*}=&({{P}_{n}^2}+{{P}_{n+1}^2}+{{P}_{n+2}^2}+{P}_{n+3}^2)\\
&\quad \quad+2\,i\,j\,({P}_{n}\,P_{n+3}-{P}_{n+1}\,{P}_{n+2}) \\
=&({P}_{2n+1}+{P}_{2n+5})+4\,i\,j\,(-1)^{n+1} \\
=&6\,{P}_{2n+3}+4\,i\,j\,(-1)^{n+1}. 
\end{array}
\end{equation*}
(2.17): By using (2.13)\\
\begin{equation*}
\begin{array}{rl}
{\,{BP}_{n}}\times{\,({BP}_{n})_i^*}+{\,{BP}_{n-1}}\times{\,({BP}_{n-1})_i^*}=&2\,[\,(Q_{2n+3}+\,{Q}_{2n+1}) \\
&\quad \quad+j\,({P}_{2n+3}+{P}_{2n+1})\,]\\=&2\,(\,8\,P_{2n+2}+\,j\,{Q}_{2n+2}).
\end{array}
\end{equation*}
(2.18): By using (2.14)\\
\begin{equation*}
\begin{array}{rl}
{\,{BP}_{n}}\times{\,({BP}_{n})_j^*}+{\,{BP}_{n-1}}\times{\,({BP}_{n-1})_j^*}=&({{P}_{n-1}^2}-{{P}_{n+3}^2}) \\
&\quad \quad+4\,i\,({P}_{n}\,Q_{n}+{Q}_{2n+2}) \\=&-12\,P_{2n+2}+4\,i\,(3\,{P}_{2n+2}) \\=&-12\,(\,P_{2n+2}-\,i\,{P}_{2n+2}).
\end{array}
\end{equation*}
(2.19): By using (2.15)\\
\begin{equation*}
\begin{array}{rl}
{\,{BP}_{n}}\times{\,({BP}_{n})_{ij}^*}+{\,{BP}_{n-1}}\times{\,({BP}_{n-1})_{ij}^*}=&6\,(\,{P}_{2n+3}+{P}_{2n+1}) \\ &+4\,i\,j\,[(-1)^{n+1}+(-1)^{n}] \\ =&6\,{Q}_{2n+2}.
\end{array}
\end{equation*}
\end{proof} 
Therefore, the norm of the bicomplex Pell number ${\,{BP}_{n}}$ is defined in three different ways as follows
\begin{equation}\label{F20}
\begin{array}{rl}
{{N}_{BP_{n}}}_{i}=&\|{BP_n\times\,{BP}_{n}^*}_i\|^2=2\,|-\,Q_{2n+3}+j\,P_{2n+3}\,|, 
\end{array} 
\end{equation}
\begin{equation}\label{F21}
\begin{array}{rl}
{{N}_{BP_{n}}}_{j}=& \|BP_n\times {\,{BP}_{n}^*}_j\|^2 \\ =& |(P_{n}^2-\,{P}_{n+1}^2+P_{n+2}^2-\,{P}_{n+3}^2)+4\,i\,({P}_{2n+3}+{P}_{n}\,{P}_{n+1})|,
\end{array} 
\end{equation}
\begin{equation}\label{F22}
\begin{array}{rl}
{{N}_{BP_{n}}}_{i\,j}=&\|BP_n\times {\,{BP}_{n}^*}_{i\,j}\|^2 = |\,6\,Q_{2n+3}+4\,i\,j\,(-1)^{n+1}\,|.
\end{array}  
\end{equation} \,
\begin{thm} 
Let ${{BP}_{n}}$ and ${{BPL}_{n}}$  be the bicomplex Pell and bicomplex Pell-Lucas numbers respectively. In this case, we can give the following relations: 
\begin{equation}\label{F23}
\begin{array}{rl}
{BP}_{m}\,{BP}_{n}+{BP}_{m+1}\,{BP}_{n+1}=&4\,(\,{Q}_{m+n+4}-i\,{Q}_{m+n+4}\\
&-j\,{P}_{m+n+4}+i\,j\,{P}_{m+n+4}\,),
\end{array}
\end{equation}
\begin{equation}\label{F24}
\begin{array}{rl}
({BP}_{n})^2=&4\,{P}_{2n+3}-4\,i\,{P}_{2n+3}+2\,j\,({P}_{2n+1}-6\,{P}_{n+1}^2) \\ &+2\,i\,j\,(\,6\,{P}_{n}\,P_{n+1}+2\,{P}_{2n+1}),
\end{array}
\end{equation}  
\begin{equation}\label{F25}
\begin{array}{rl}
({BP}_{n})^2+({BP}_{n+1})^2=& 4\,(\,{Q}_{2n+4}-i\,{Q}_{2n+4}-j\,P_{2n+4} \\ &+i\,j\,P_{2n+4}\,),
\end{array}
\end{equation} \,
\begin{equation}\label{F26}
\begin{array}{rl}
({BP}_{n+1})^2-({BP}_{n-1})^2=&-4\,(\,{P}_{2n+1}+2\,i\,{Q}_{2n+3}+2\,j\,P_{2n+3} \\ &+2\,i\,j\,P_{2n+3}\,) 
\end{array}
\end{equation} 
\begin{equation}\label{F27}
\begin{array}{rl}
\quad {BP}_{n}-i\,{BP}_{n+1}+j\,{BP}_{n+2}-i\,j\,{BP}_{n+3}=&4\,(-4\,{P}_{n+3}+j\,q_{n+3}\,),
\end{array}
\end{equation} \\
\begin{equation}\label{F28}
\begin{array}{rl}
\quad {BP}_{n}-i\,{BP}_{n+1}-j\,{BP}_{n+2}-i\,j\,{BP}_{n+3}=&2\,(\,{q}_{n+1}-P_{n+5}+i\,P_{n+5} \\ &+j\,P_{n+4}-i\,j\,\,P_{n+3}\,).
\end{array}
\end{equation}
\begin{proof}
(2.23): By the equation (2.1) we get,
\begin{equation*}
\begin{array}{rl}
{BP}_{m}\,{BP}_{n}+{BP}_{m+1}\,{BP}_{n+1}=&({P}_{m+n+1}-{P}_{m+n+3}-{P}_{m+n+5} \\
&\quad \quad+{P}_{m+n+7})\\
&\quad \quad+2\,i\,({P}_{m+n+2}-{P}_{m+n+6})\\
&\quad \quad+2\,j\,({P}_{m+n+3}-{P}_{m+n+5}) \\ 
&\quad \quad+2\,i\,j\,(2\,{P}_{m+n+4}) \\
=& 4\,({Q}_{m+n+4}-i\,{Q}_{m+n+4}-j\,{P}_{m+n+4} \\
&\quad \quad+i\,j\,{P}_{m+n+4}).  
\end{array}
\end{equation*}
(2.24): By the equation (2.1) we get,
\begin{equation*}
\begin{array}{rl}
({BP}_{n})^2=&({{P}_{n}^2}-{{P}_{n+1}^2}-{{P}_{n+2}^2}+{P}_{n+3}^2)+2\,i\,({P}_{n}\,P_{n+1}-{P}_{n+2}\,{P}_{n+3}) \\ 
&\quad+2\,j\,({P}_{n}\,P_{n+2}-{P}_{n+1}\,{P}_{n+3})+2\,i\,j\,({P}_{n}\,P_{n+3}+{P}_{n+1}\,{P}_{n+2}) \\ 
=&4\,{P}_{2n+3}-4\,i\,{P}_{2n+3}+2\,j\,({P}_{2n+1}-6\,{P}_{n+1}^2) \\
&\quad+2\,i\,j\,(\,6\,{P}_{n}\,P_{n+1}+2\,{P}_{2n+1}).
\end{array}
\end{equation*}
(2.25): By the equations (2.1) and (2.23) we get,
\begin{equation*}
\begin{array}{rl}
({BP}_{n})^2+\,({BP}_{n+1})^2=&({{P}_{n}^2}-{{P}_{n+2}^2}+{{P}_{n+4}^2}-{P}_{n+2}^2) \\
&\quad+2\,i\,({P}_{2n+2}-{P}_{2n+6})+2\,j\,({P}_{2n+3}-{P}_{2n+5}) \\
&\quad+2\,i\,j\,(2\,{P}_{2n+4}) \\
=&4\,(\,{Q}_{2n+4}-i\,{Q}_{2n+4}-j\,P_{2n+4}+\,i\,j\,P_{2n+4}\,).
\end{array}
\end{equation*}
(2.26): By the equations (2.1) and (2.23) we get,
\begin{equation*}
\begin{array}{rl}
({BP}_{n+1})^2-\,({BP}_{n-1})^2=&({{P}_{n+1}^2}-{{P}_{n-1}^2}+{{P}_{n}^2}-{P}_{n+2}^2) \\
&\quad+2\,i\,[2\,({P}_{2n+1}-{P}_{2n+5})] \\ 
&\quad+2\,j\,({P}_{2n+3}-5{P}_{2n+3}) \\
&\quad+2\,i\,j\,[4\,({q}_{2n+2}+{P}_{2n+2})] \\
=&2\,({P}_{2n}-{P}_{2n+2})+2\,i\,(-4\,{Q}_{2n+3}) \\
&\quad+2\,j\,(-4\,{P}_{2n+3})+2\,i\,j\,(4\,{P}_{2n+3}) \\
=&-4\,(\,{P}_{2n+1}+2\,i\,{Q}_{2n+3}+2\,j\,P_{2n+3} \\
&\quad+2\,i\,j\,P_{2n+3}\,)
\end{array}
\end{equation*}
(2.27): By the equation (2.1) we get,
\begin{equation*}
\begin{array}{rl}
{BP}_{n}-i\,{BP}_{n+1}-j\,{BP}_{n+2}-i\,j\,{BP}_{n+3}=&({P}_{n}+{P}_{n+2}+{P}_{n+4}-{P}_{n+6}) \\ 
&+2\,i\,({P}_{n+5})+2\,j\,({P}_{n+4}) \\
&-2\,i\,j\,({P}_{n+3}) \\
=&-(4\,{P}_{n+1}+{P}_{n})+2\,i\,{P}_{n+5} \\ &+2\,j\,{P}_{n+4}-2\,i\,j\,{P}_{n+3}.\,
\end{array}
\end{equation*}
(2.28): By the equation (2.1) we get,
\begin{equation*}
\begin{array}{rl}
{BP}_{n}-i\,{BP}_{n+1}-j\,{BP}_{n+2}-i\,j\,{BP}_{n+3}=&({P}_{n}+{P}_{n+2}+{P}_{n+4}-{P}_{n+6}) \\ 
&+2\,i\,({P}_{n+5})+2\,j\,({P}_{n+4}) \\
&-2\,i\,j\,({P}_{n+3}) \\
=&-(4\,{P}_{n+1}+{P}_{n})+2\,i\,{P}_{n+5} \\ &+2\,j\,{P}_{n+4}-2\,i\,j\,{P}_{n+3}.\,
\end{array}
\end{equation*}.
\end{proof}
\end{thm}
\begin{thm} (\textit{d'Ocagne's identity}). For $n,m\geq0$  d'Ocagne's identity for bicomplex Pell numbers $BP_n$ and $BP_m$ is given by
\begin{equation}\label{F29}
\begin{array}{rl}
{BP}_{m}\,{BP}_{n+1}-{BP}_{m+1}\,{BP}_{n}=&12\,(-1)^{n}\,{P}_{m-n}\,(\,j\,+i\,j\,),
\end{array}
\end{equation}
\end{thm}
\begin{proof}
(2.29): By the equation (2.1) we get,
\begin{equation*}
\begin{array}{rl}
{BP}_{m}\,{BP}_{n+1}-{BP}_{m+1}\,{BP}_{n}=&(-1)^{n}{P}_{m-n}\,(\,0\,)\\
&\quad+i\,(-1)^{n}({P}_{m-n-1}\,(\,0\,)\\
&\quad+2\,j\,(-1)^{n}(\,{P}_{m-n-2}+{P}_{m-n+2})\\ 
&\quad+i\,j\,(-1)^{n}[\,(-{P}_{m-n-3}+{P}_{m-n+3}\\
&\quad \quad +{P}_{m-n-1}-{P}_{m-n+1}\,)\,]\\
=&2\,j\,(-1)^{n}\,\,(\,6\,{P}_{m-n}\,)\\ 
&\quad+i\,j\,(-1)^{n}\,6\,(\,{P}_{m-n-1}-{P}_{m-n+1}\,)\\
=& 12\,(-1)^{n}\,{P}_{m-n}\,(\,j+\,+i\,j\,).  
\end{array}
\end{equation*}
\end{proof}

\begin{thm} 
Let ${\,{BP}_{n}}$ and ${{BPL}_{n}}$ be the bicomplex Pell number and the bicomplex Pell-Lucas numbers respectively. The following relations are satisfied

\begin{equation}\label{F30}
\begin{aligned}
{\,{BP}_{n+1}}+{\,{BP}_{n-1}}={BPL}_{n}\,, \\
\end{aligned}
\end{equation}
\begin{equation}\label{F31}
{BP}_{n+1}-{BP}_{n-1}=2\,{BP}_{n},
\end{equation}
\begin{equation}\label{F32}
\begin{aligned}
{\,{BP}_{n+2}}+{\,{BP}_{n-2}}=6\,{BP}_{n}.
\end{aligned}
\end{equation}
\begin{equation}\label{F33}
{BP}_{n+2}-{BP}_{n-2}=2\,{BPL}_{n},
\end{equation}
\begin{equation}\label{F34}
{BP}_{n+1}+{BP}_{n}=\frac{1}{2}\,{BPL}_{n+1},
\end{equation} 
\begin{equation}\label{F35}
{BP}_{n+1}-{BP}_{n}=\frac{1}{2}\,{BPL}_{n},
\end{equation} 
\begin{equation}\label{F36}
{BPL}_{n+1}+{BPL}_{n-1}=4\,{BP}_{n},
\end{equation} 
\begin{equation}\label{F37}
{BPL}_{n+1}-{BPL}_{n-1}=2\,{BPL}_{n},
\end{equation} 
\begin{equation}\label{F38}
{BPL}_{n+2}+{BPL}_{n-2}=6\,{BPL}_{n},
\end{equation}
\begin{equation}\label{F39}
{BPL}_{n+2}-{BPL}_{n-2}=8\,{BP}_{n},
\end{equation} 
\begin{equation}\label{F40}
{BPL}_{n+1}+{BPL}_{n}=4\,{BP}_{n+1},
\end{equation}
\begin{equation}\label{F41}
{BPL}_{n+1}-{BPL}_{n}=4\,{BP}_{n},
\end{equation} 
\end{thm}
\begin{proof}

(2.30):By the equation (2.1) we get,
\begin{equation*}
\begin{aligned}
\begin{array}{rl}
{\,{BP}_{n+1}}+{\,{BP}_{n-1}}=&({{P}_{n+1}}+{{P}_{n-1}})+i\,({{P}_{n+2}}+{{P}_{n}}) \\
&\quad+j\,({{P}_{n+3}}+{{P}_{n+1}})+i\,j\,({{P}_{n+4}}+{{P}_{n+2}}) \\
=&(\,{Q}_{n}+i\,{Q}_{n+1}+j\,{Q}_{n+2}+i\,j\,{Q}_{n+3}\,) \\ = & {BPL}_{n}\,, \\
\end{array}
\end{aligned}
\end{equation*}
(2.31): By the equation (2.1) we get,
\begin{equation*}
\begin{aligned}
\begin{array}{rl}
{BP}_{n+1}-{BP}_{n-1}=&({P}_{n+1}-{P}_{n-1})+i\,({P}_{n+2}-{P}_{n}) \\ 
&\quad+j\,({P}_{n+3}-{P}_{n+1})+i\,j\,({P}_{n+4}-{P}_{n+2}) \\
=&2\,({P}_{n}+i\,{P}_{n+1}+j\,{P}_{n+2}+i\,j\,{P}_{n+3}) \\
=&2\,{BP}_{n}. 
\end{array}
\end{aligned}
\end{equation*}
(2.32): By the equation (2.1) we get,
\begin{equation*}
\begin{aligned}
\begin{array}{rl}
{BP}_{n+2}+{BP}_{n-2}=&({P}_{n+2}+{P}_{n-2})+i\,({P}_{n+3}+{P}_{n-1}) \\ 
&\quad+j\,({P}_{n+4}+{P}_{n})+i\,j\,({P}_{n+5}+{P}_{n+1})  \\
=&6\,({P}_{n}+i\,{P}_{n+1}+j\,{P}_{n+2}+i\,j\,{P}_{n+3}) \\
=&6\,{BP}_{n}. 
\end{array}
\end{aligned}
\end{equation*}
(2.33): By the equation (2.1) we get,
\begin{equation*}
\begin{aligned}
\begin{array}{rl}
{\,{BP}_{n+2}}-{\,{BP}_{n-2}}=&({{P}_{n+2}}-{{P}_{n-2}})+i\,({{P}_{n+3}}-{{P}_{n-1}}) \\ 
&\quad+j\,({{P}_{n+4}}-{{P}_{n}})+i\,j\,({{P}_{n+5}}-{{P}_{n+1}}) \\
=&2\,({Q}_{n}+i\,{Q}_{n+1}+j\,{Q}_{n+2}+i\,j\,{Q}_{n+3}) \\ =& 2\, {BPL}_{n}\,. 
\end{array}
\end{aligned}
\end{equation*}
(2.34): By the equation (2.1) we get,
\begin{equation*}
\begin{aligned}
\begin{array}{rl}
{BP}_{n+1}+{BP}_{n}=&({P}_{n+1}+{P}_{n})+i\,({P}_{n+2}+{P}_{n+1}) \\ 
&\quad+j\,({P}_{n+3}+P_{n+2})+i\,j\,({P}_{n+4}+{P}_{n+3}) \\
=&({q}_{n+1}+i\,{q}_{n+2}+j\,{q}_{n+3}+i\,j\,{q}_{n+4}) \\
=&\frac{1}{2}\,({Q}_{n+1}+i\,{Q}_{n+2}+j\,{Q}_{n+3}+i\,j\,{Q}_{n+4}) \\
=&\frac{1}{2}\,{BPL}_{n+1} 
\end{array}
\end{aligned}
\end{equation*}
where the property (1.17) of the modified Pell number is used.\\
(2.35): By the equation (2.1) we get,
\begin{equation*}
\begin{aligned}
\begin{array}{rl}
{BP}_{n+1}-{BP}_{n}=&({P}_{n+1}-{P}_{n})+i\,({P}_{n+2}-{P}_{n+1}) \\
&\quad+j\,({P}_{n+3}-P_{n+2})+i\,j\,({P}_{n+4}-{P}_{n+3}) \\
=&({q}_{n}+i\,{q}_{n+1}+j\,{q}_{n+2}+i\,j\,{q}_{n+3}) \\
=&\frac{1}{2}\,({Q}_{n}+i\,{Q}_{n+1}+j\,{Q}_{n+2}+i\,j\,{Q}_{n+3}) \\
=&\frac{1}{2}\,{BPL}_{n} 
\end{array}
\end{aligned}
\end{equation*}
where the property (1.17) of the modified Pell number is used.\\
(2.36): By the equation (2.2) we get,
\begin{equation*}
\begin{aligned}
\begin{array}{rl}
BPL_{n+1}+{BPL}_{n-1}=&({Q}_{n+1}+{Q}_{n-1})+i\,({Q}_{n+2}+{Q}_{n}) \\
&+j\,({Q}_{n+3}+{Q}_{n+1})+i\,j\,({Q}_{n+4}+{Q}_{n+2}) \\
=&4\,({P}_{n}+i\,{P}_{n+1}+j\,{P}_{n+2}+i\,j\,{P}_{n+3}) \\
=&4\,{BP}_{n}.
\end{array}
\end{aligned}
\end{equation*}
(2.37): By the equation (2.2) we get,
\begin{equation*}
\begin{aligned}
\begin{array}{rl}
BPL_{n+1}-{BPL}_{n-1}=&({Q}_{n+1}-{Q}_{n-1})+i\,({Q}_{n+2}-{Q}_{n}) \\
&+j\,({Q}_{n+3}-{Q}_{n+1})+i\,j\,({Q}_{n+4}-{Q}_{n+2}) \\
=&2\,({Q}_{n}+i\,{Q}_{n+1}+j\,{Q}_{n+2}+i\,j\,{Q}_{n+3}) \\
=&2\,{BPL}_{n}
\end{array}
\end{aligned}
\end{equation*}
(2.38): By the equation (2.2) we get,
\begin{equation*}
\begin{array}{rl}
BPL_{n+2}+{BPL}_{n-2}=&({Q}_{n+2}+{Q}_{n-2})+i\,({Q}_{n+3}+{Q}_{n-1}) \\
&+j\,({Q}_{n+4}+{Q}_{n})+i\,j\,({Q}_{n+5}+{Q}_{n+1}) \\
=&6\,({Q}_{n}+i\,{Q}_{n+1}+j\,{Q}_{n+2}+i\,j\,{Q}_{n+3}) \\
=&6\,{BPL}_{n}. 
\end{array}
\end{equation*}
(2.39): By the equation (2.2) we get,
\begin{equation*}
\begin{aligned}
\begin{array}{rl}
BPL_{n+2}-{BPL}_{n-2}=&({Q}_{n+2}-{Q}_{n-2})+i\,({Q}_{n+3}-{Q}_{n-1}) \\
&+j\,({Q}_{n+4}-{Q}_{n})+i\,j\,({Q}_{n+5}-{Q}_{n+1}) \\
=&8\,({P}_{n}+i\,{P}_{n+1}+j\,{P}_{n+2}+i\,j\,{P}_{n+3}) \\
=&8\,{BP}_{n}. 
\end{array}
\end{aligned}
\end{equation*}
(2.40): By the equation (2.2) we get,
\begin{equation*}
\begin{aligned}
\begin{array}{rl}
{BPL}_{n+1}+{BPL}_{n}=&({Q}_{n+1}+{Q}_{n})+i\,({Q}_{n+2}+{Q}_{n+1}) \\
&+j\,({Q}_{n+3}+{Q}_{n+2})+i\,j\,({Q}_{n+4}+{Q}_{n+3}) \\
=&4\,{P}_{n+1}+i\,{P}_{n+2}+j\,{P}_{n+3}+i\,j\,{P}_{n+4} \\
=&4\,{BP}_{n+1}.
\end{array}
\end{aligned}
\end{equation*}
(2.41): By the equation (2.2) we get,
\begin{equation*}
\begin{aligned}
\begin{array}{rl}
{BPL}_{n+1}-{BPL}_{n}=&({Q}_{n+1}-{Q}_{n})+i\,({Q}_{n+2}-{Q}_{n+1}) \\
&+j\,({Q}_{n+3}-{Q}_{n+2})+i\,j\,({Q}_{n+4}-{Q}_{n+3}) \\
=&4\,{P}_{n}+i\,{P}_{n+1}+j\,{P}_{n+2}+i\,j\,{P}_{n+3} \\
=&4\,{BP}_{n}.
\end{array}
\end{aligned}
\end{equation*}
\end{proof}
\begin{thm} 
If ${{BP}_{n}}$ and ${{BPL}_{n}}$  are bicomplex Pell and bicomplex Pell-Lucas numbers respectively, then for $n\geq0$, the identities of negabicomplex Pell and negabicomplex Pell-Lucas numbers are 
\begin{equation}\label{F42}
\begin{array}{rl}
{BP}_{-n}=&(-1)^{n+1}\,{BP}_{n}+(-1)^{n}\,{Q}_{n}\,(\,i\,+2\,j\,+5\,i\,j\,),
\end{array}
\end{equation}
and
\begin{equation}\label{F43}
\begin{array}{rl}
{BPL}_{-n}=&(-1)^{n}\,{BPL}_{n}+8\,(-1)^{n+1}\,{P}_{n}\,(\,i\,+2\,j\,+5\,i\,j\,),
\end{array}
\end{equation}
\end{thm}
\begin{proof}
(2.42): By using the identity of negapell numbers \,${P}_{-n}=(-1)^{n+1}\,{P}_{n}$ \, we get \\
\begin{equation*}
\begin{array}{rl}
{\,{BP}_{-n}}=&{P}_{-n}+i\,{P}_{-n+1}+j\,{P}_{-n+2}+i\,j\,{P}_{-n+3} \\
=&{P}_{-n}+i\,{P}_{-(n-1)}+j\,{P}_{-(n-2)}+i\,j\,{P}_{-(n-3)} \\
=&(-1)^{n+1}\,{P}_{n}+i\,(-1)^{n}\,{P}_{n-1}+j\,(-1)^{n-1}\,{P}_{n-2} \\
&\quad \quad+i\,j\,(-1)^{n-2}\,{P}_{n-3} \\
=&(-1)^{n+1}\,({P}_{n}+i\,{P}_{n+1}+j\,{P}_{n+2}+i\,j\,{P}_{n+3}) \\
&\quad \quad-i\,(-1)^{n+1}\,{P}_{n+1}-j\,(-1)^{n+1}\,{P}_{n+2}-i\,j\,(-1)^{n+1}\,{P}_{n+3} \\
&\quad \quad+i\,(-1)^{n}\,{P}_{n-1}+j\,(-1)^{n+1}\,{P}_{n-2}+i\,j\,(-1)^{n}\,{P}_{n-3} \\
=&(-1)^{n+1}\,{BP}_{n}+(-1)^{n}\,({P}_{n+1}+{P}_{n-1})\,i \\
&\quad \quad+(-1)^{n}\,({P}_{n+2}-{P}_{n-2})\,j+(-1)^{n}\,({P}_{n+3}+{P}_{n-3})\,i\,j \\
=&(-1)^{n+1}\,{BP}_{n}+(-1)^{n}\,{Q}_{n}\,(\,i\,+2\,j\,+5\,i\,j\,)
\end{array}
\end{equation*}
(2.43): By using the identity of negapell-Lucas numbers \, ${Q}_{-n}=(-1)^{n}\,{Q}_{n}$ \,  we get \\
\begin{equation*}
\begin{array}{rl}
{\,{BPL}_{-n}}=&{Q}_{-n}+i\,{Q}_{-n+1}+j\,{Q}_{-n+2}+i\,j\,{Q}_{-n+3} \\
=&{Q}_{-n}+i\,{Q}_{-(n-1)}+j\,{Q}_{-(n-2)}+i\,j\,{Q}_{-(n-3)} \\
=&(-1)^{n}\,{Q}_{n}+i\,(-1)^{n-1}\,{Q}_{n-1}+j\,(-1)^{n-2}\,{Q}_{n-2} \\
&\quad \quad+i\,j\,(-1)^{n-3}\,{Q}_{n-3} \\
=&(-1)^{n+1}\,({Q}_{n}+i\,{Q}_{n+1}+j\,{Q}_{n+2}+i\,j\,{Q}_{n+3}) \\
&\quad \quad-i\,(-1)^{n}\,{Q}_{n+1}-j\,(-1)^{n}\,{Q}_{n+2} \\
&\quad \quad-i\,j\,(-1)^{n}\,{Q}_{n+3} \\
&\quad \quad+i\,(-1)^{n-1}\,{Q}_{n-1}+j\,(-1)^{n}\,{Q}_{n-2} \\
&\quad \quad+i\,j\,(-1)^{n-1}\,{Q}_{n-3} \\
=&(-1)^{n+1}\,{BPL}_{n}+(-1)^{n+1}\,({Q}_{n+1}+{Q}_{n-1})\,i \\
&\quad \quad+(-1)^{n+1}\,({Q}_{n+2}-{Q}_{n-2})\,j \\
&\quad \quad+(-1)^{n+1}\,({Q}_{n+3}+{Q}_{n-3})\,i\,j \\
=&(-1)^{n}\,{BPL}_{n}+8\,(-1)^{n+1}\,{P}_{n}\,(\,i\,+2\,j\,+5\,i\,j\,)
\end{array}
\end{equation*}
\end{proof}
\begin{thm} \textit{Binet's Formula}. 
Let ${\,{BP}_{n}}$ and ${\,{BPL}_{n}}$ be the bicomplex Pell and bicomplex Pell-Lucas numbers respectively. For $n\ge 1$, 
Binet's formula for these numbers are as follows:
\begin{equation}\label{F44}
{\,{BP}_{n}}=\frac{1}{\alpha -\beta }( \hat{\alpha }\,\,{\alpha }^{n}-\hat{\beta \,}\,{{\beta }^{n}}) \,
\end{equation}
and
\begin{equation}\label{F45}
{\,{BPL}_{n}}= \,\hat{\alpha }\,\,{{\alpha }^{n}}+\hat{\beta \,}\,{{\beta }^{n}} 
\end{equation}
where \, $\hat{\alpha }=1+i\,{\alpha}+j\,{\alpha}^2+i\,j\,{\alpha}^3$, \,\,\,$\alpha  = {1 + \sqrt 2 }$ \,and \,
$\hat{\beta }=1+i\,{\beta}+j\,{\beta}^2+i\,j\,{\beta}^3$, \,\, \,$\beta  = {1 - \sqrt 2 }$.
\end{thm}
\begin{proof} 
(2.44):
\begin{equation*}
\begin{array}{rl} 
{BP}_{n}=&{P}_{n}+i\,{P}_{n+1}+j\,{P}_{n+2}+i\,j\,{P}_{n+3}\\=&\frac{\alpha^n-\,\beta^n}{\alpha -\beta }+i\,\frac{\alpha^{n+1}-\,\beta^{n+1}}{\alpha -\beta }+j\,\frac{\alpha^{n+2}-\,\beta^{n+2}}{\alpha -\beta }+i\,j\,\frac{\alpha^{n+3}-\,\beta^{n+3}}{\alpha -\beta } \\
=&\frac{{\alpha }^n\,(1+i\,{\alpha}+j\,{\alpha}^2+i\,j\,{\alpha}^3\,)-{\beta }^{n}\,(1+i\,{\beta}+j\,{\beta}^2+i\,j\,{\beta}^3\,)}{\alpha -\beta } \\
=&\frac{ \,\hat{\alpha }\,\,{\alpha }^n-\hat{\beta}\,{\beta }^n }{\alpha -\beta}
\end{array}
\end{equation*}
and
(2.45):
\begin{equation*}
\begin{array}{rl} 
{BPL}_{n}=&{Q}_{n}+i\,{Q}_{n+1}+j\,{Q}_{n+2}+i\,j\,{Q}_{n+3}\\
=&{\alpha }^n+{\beta }^n +i\,({\alpha }^{n+1}+\,{\beta }^{n+1})+j\,({\alpha }^{n+2}+\,{\beta }^{n+2})+i\,j\,({\alpha }^{n+3}+\,{\beta }^{n+3}) \\
=&{\alpha }^{n}\,(1+i\,{\alpha}+j\,{\alpha}^{2}+i\,j\,{\alpha}^{3}\,)+{\beta }^{n}\,(1+i\,{\beta}+j\,{\beta}^2+i\,j\,{\beta}^3\,) \\
=&\hat{\alpha }\,\,{\alpha }^{n}+\hat{\beta \,}\,{\beta }^{n}.
\end{array}
\end{equation*}
Binet's formula of the bicomplex Pell number is the same as Binet's formula of the Pell number \cite{C}.
\end{proof}
\begin{thm} \textit{Cassini's Identity} 
Let ${\,{BP}_{n}}$ and ${\,{BPL}_{n}}$ be the bicomplex Pell and bicomplex Pell-Lucas numbers, respectively. For $n\ge 1$,  Cassini's identities for ${\,{BP}_{n}}$ and ${\,{BPL}_{n}}$  are as follows: \\
\begin{equation}\label{F46}
{\,{BP}_{n-1}}\,{\,{BP}_{n+1}}-{BP}_{n}^2=12\,(-1)^{n}\,(j+i\,j) \\
\end{equation}
and
\begin{equation}\label{F47}
{\,{BPL}_{n-1}}\,{\,{BPL}_{n+1}}-{BPL}_{n}^2=8.12\,(-1)^{n+1}\,(j+i\,j). \\
\end{equation}
\end{thm} 
\begin{proof}
(2.46): By using (2.1) we get \\
\begin{equation*}
\begin{array}{rl}
{\,{BP}_{n-1}}\,{\,{BP}_{n+1}}\,-(\,{BP}_{n})^2=&({P}_{n-1}+i\,{P}_{n}+j\,{P}_{n+1}+i\,j\,{P}_{n+2})\,\\&({P}_{n+1}+i\,{P}_{n+2}+j\,{P}_{n+3}+i\,j\,{P}_{n+4}) \\
&-[\,{P}_{n}+i\,{P}_{n+1}+j\,{P}_{n+2}+i\,j\,{P}_{n+3}\,]^2 \\
=&\,[\,({P}_{n-1}{P}_{n+1}-{P}_{n}^2) \\
&\quad \quad-({P}_{n}{P}_{n+2}+{P}_{n+1}^2) \\ 
&\quad \quad-({P}_{n+1}{P}_{n+3}-{P}_{n+2}^2) \\
&\quad \quad+({P}_{n+2}{P}_{n+4}-{P}_{n+3}^2)\,] \\
&+i\,[\,({P}_{n+2}{P}_{n-1}-{P}_{n+1}{P}_{n}) \\
&\quad \quad-({P}_{n+4}{P}_{n+1}-{P}_{n+3}{P}_{n+2})\,] \\
&+j\,[\,({P}_{n+1}{P}_{n+1}-{P}_{n}{P}_{n+2}) \\
&\quad \quad-({P}_{n+2}{P}_{n+2}-{P}_{n+1}{P}_{n+3}) \\ 
& \quad \quad +({P}_{n+3}{P}_{n-1}-{P}_{n+2}{P}_{n}) \\
&\quad \quad-({P}_{n+4}{P}_{n}-{P}_{n+3}{P}_{n+1})\,] \\
&+i\,j\,({P}_{n+4}{P}_{n-1}-{P}_{n+3}{P}_{n}) \\
=&12\,(-1)^{n}\,(j+i\,j). 
\end{array}
\end{equation*}
(2.47): By using (2.2) we get 
\begin{equation*}
\begin{array}{rl}
{\,{BPL}_{n-1}}\,{\,{BPL}_{n+1}}\,-(\,{BPL}_{n})^2=&({Q}_{n-1}+i\,{Q}_{n}+j\,{Q}_{n+1}+i\,j\,{Q}_{n+2})\,\\&({Q}_{n+1}+i\,{Q}_{n+2}+j\,{Q}_{n+3}+i\,j\,{Q}_{n+4}) \\
&-[{Q}_{n}+i\,{Q}_{n+1}+j\,{Q}_{n+2}+i\,j\,{Q}_{n+3}]^2 \\
=&\,[\,({Q}_{n-1}{Q}_{n+1}-{Q}_{n}^2) \\
&\quad \quad+({Q}_{n+1}^2-{Q}_{n+2}{Q}_{n}) \\ 
&\quad \quad+({Q}_{n+2}^2-{Q}_{n+3}{Q}_{n+1})\\
&\quad \quad+({Q}_{n+4}{Q}_{n+2}-{Q}_{n+3}^2)\,] \\
&+i\,[({Q}_{n+2}{Q}_{n-1}-{Q}_{n+1}{Q}_{n})\\
&\quad \quad+({Q}_{n+3}{Q}_{n+2}-{Q}_{n+4}{Q}_{n+1})\,] \\
&+j\,[\,({Q}_{n+1}{Q}_{n+1}-{Q}_{n}{Q}_{n+2})\\
&\quad \quad+({Q}_{n+1}{Q}_{n+3}-({Q}_{n+2}{Q}_{n+2})\\ 
&\quad \quad+({Q}_{n+3}{Q}_{n-1}-{Q}_{n+2}{Q}_{n})\\
&\quad \quad+({Q}_{n+3}({Q}_{n+1}-{Q}_{n+4}{Q}_{n})\,] \\
&+i\,j\,({Q}_{n+4}{Q}_{n-1}-{Q}_{n+3}{Q}_{n}) \\
=&8.12\,(-1)^{n+1}\,(\,j+i\,j). 
\end{array}
\end{equation*}
where the identities of the Pell and Pell-Lucas numbers \,\,
${P}_{m}\,{P}_{n+1}-{P}_{m+1}\,{P}_{n}=(-1)^{n}{P}_{m-n}$ \,\,and \,\,${Q}_{m}\,{Q}_{n+1}-{Q}_{m+1}{Q}_{n}=8\,(-1)^{n+1}\,{P}_{m-n}$\, are used.
\end{proof} 
\begin{thm} \textit{Catalan's Identity}. Let ${{BP}_{n}}$\,and \, ${{BPL}_{n}}$ be the bicomplex Pell and bicomplex Pell-Lucas numbers, repectively. For $n\ge 1$, Catalan's identities for ${BP}_{n}$ and ${BPL}_{n}$ are as follows
\begin{equation}\label{F48}
({BP}_{n})^2-{BP}_{n+r}\,\,{BP}_{n-r}=12\,(-1)^{n-r}\,{{P}_{r}}^2\,(\,j+i\,j\,),
\end{equation}
\begin{equation}\label{F49}
({BPL}_{n})^2-{BPL}_{n+r}\,\,{BPL}_{n-r}=8\,.\,12\,(-1)^{n-r}\,{{P}_{r}}^2\,(\,j+i\,j\,). \\
\end{equation}
respectively. 
\end{thm}
\begin{proof}
(2.48): By using (2.1) we get 
\begin{equation*}
\begin{array}{rl}
{BP}_{n}^2-{BP}_{n+r}\,\,{BP}_{n-r}=&[\,({P}_{n}^2-{P}_{n+r}\,{P}_{n-r}) \\
&\quad-({P}_{n+1}^2-{P}_{n+r+1}\,{P}_{n-r+1}) \\ 
&\quad-({P}_{n+2}^2-{P}_{n+r+2}\,{P}_{n-r+2}) \\
&\quad+({P}_{n+3}^2-{P}_{n+r+3}\,{P}_{n-r+3})\,] \\ 
&+i\,[\,({P}_{n}\,{P}_{n+1}-{P}_{n+r}\,{P}_{n-r+1}) \\
&\quad-({P}_{n+2}\,{P}_{n+3}-{P}_{n+r+2}\,{P}_{n-r+3})\\ 
&\quad+({P}_{n+1}\,{P}_{n}-{P}_{n+r+1}\,{P}_{n-r}) \\ 
&\quad-({P}_{n+3}\,{P}_{n+2}-{P}_{n+r+3}\,{P}_{n-r+2})\,] \\ 
&+j\,[\,({P}_{n}\,{P}_{n+2}-{P}_{n+r}\,{P}_{n-r+2}) \\
&\quad-({P}_{n+1}\,{P}_{n+3}-{P}_{n+r+1}\,{P}_{n-r+3}) \\ 
&\quad+({P}_{n+2}\,{P}_{n}-{P}_{n+r+2}\,{P}_{n-r}) \\ 
&\quad-({P}_{n+3}\,{P}_{n+1}-{P}_{n+r+3}\,{P}_{n-r+1})\,] \\ 
&+i\,j\,[\,({P}_{n}\,{P}_{n+3}-{P}_{n+r}\,{P}_{n-r+3}) \\
&\quad+({P}_{n+1}\,{P}_{n+2}-{P}_{n+r+1}\,{P}_{n-r+2}) \\ 
&\quad+({P}_{n+3}\,{P}_{n}-{P}_{n+r+3}\,{P}_{n-r}) \\ 
&\quad+({P}_{n+2}\,{P}_{n+1}-{P}_{n+r+2}\,{P}_{n-r+1})\,] \\
=& (-1)^{n-r}\,{P}_{r}^2\,(0+0i+12\,j+12\,\,i\,j) \\
=& 12\,(-1)^{n-r}\,{P}_{r}^2\,(j+i\,j). \\
\end{array}
\end{equation*} 
(2.49): By using (2.2) we get 
\begin{equation*}
\begin{array}{rl}
({BPL}_{n})^2-{BPL}_{n+r}\,{BPL}_{n-r}=&[\,({Q}_{n}^2-{Q}_{n+r}\,{Q}_{n-r}) \\
&\quad-({Q}_{n+1}^2-{Q}_{n+r+1}\,{Q}_{n-r+1}) \\ 
&\quad-({Q}_{n+2}^2-{Q}_{n+r+2}\,{Q}_{n-r+2}) \\
&\quad+({Q}_{n+3}^2-{Q}_{n+r+3}\,{Q}_{n-r+3})\,] \\ 
&+i\,[\,({Q}_{n}\,{Q}_{n+1}-{Q}_{n+r}\,{Q}_{n-r+1}) \\
&\quad-({Q}_{n+2}\,{Q}_{n+3}-{Q}_{n+r+2}\,{Q}_{n-r+3})\\ 
&\quad+({Q}_{n+1}\,{Q}_{n}-{Q}_{n+r+1}\,{Q}_{n-r}) \\ 
&\quad-({Q}_{n+3}\,{Q}_{n+2}-{Q}_{n+r+3}\,{Q}_{n-r+2})\,] \\ 
\end{array}
\end{equation*}

\begin{equation*}
\begin{array}{rl}
&\quad\quad\quad\quad\quad\quad\quad\quad\quad\quad\quad+j\,[\,({Q}_{n}\,{Q}_{n+2}-{Q}_{n+r}\,{Q}_{n-r+2}) \\
&\quad\quad\quad\quad\quad\quad\quad\quad\quad\quad\quad\quad-({Q}_{n+1}\,{Q}_{n+3}-{Q}_{n+r+1}\,{Q}_{n-r+3}) \\ 
&\quad\quad\quad\quad\quad\quad\quad\quad\quad\quad\quad\quad+({Q}_{n+2}\,{Q}_{n}-{Q}_{n+r+2}\,{Q}_{n-r}) \\ 
&\quad\quad\quad\quad\quad\quad\quad\quad\quad\quad\quad\quad-({Q}_{n+3}\,{Q}_{n+1}-{Q}_{n+r+3}\,{Q}_{n-r+1})\,] \\ 
&\quad\quad\quad\quad\quad\quad\quad\quad\quad\quad\quad+i\,j\,[\,({Q}_{n}\,{Q}_{n+3}-{Q}_{n+r}\,{Q}_{n-r+3}) \\
&\quad\quad\quad\quad\quad\quad\quad\quad\quad\quad\quad\quad+({Q}_{n+1}\,{Q}_{n+2}-{Q}_{n+r+1}\,{Q}_{n-r+2}) \\ 
&\quad\quad\quad\quad\quad\quad\quad\quad\quad\quad\quad\quad+({Q}_{n+3}\,{Q}_{n}-{Q}_{n+r+3}\,{Q}_{n-r}) \\ 
&\quad\quad\quad\quad\quad\quad\quad\quad\quad\quad\quad\quad+({Q}_{n+2}\,{Q}_{n+1}-{Q}_{n+r+2}\,{Q}_{n-r+1})\,] \\
\end{array}
\end{equation*}
\begin{equation*}
\begin{array}{rl}
\quad\quad\quad\quad\quad\quad\quad\quad\quad\quad\quad\quad=& 8\,(-1)^{n-r}\,{P}_{r}^2\,(0+0i+12\,j+12\,\,i\,j) \\
=& 8.12\,(-1)^{n-r}\,{P}_{r}^2\,(j+i\,j). \\
\end{array}
\end{equation*}\\
where the identities of the Pell and Pell-Lucas numbers are used as follows,\\
\begin{equation*}
\begin{array}{rl}
{P}_{m}{P}_{n}-{P}_{m+r}{P}_{n-r}=&(-1)^{n-r}{P}_{m+r-n}\,{P}_{r},\\
{P}_{n}\,{P}_{n}-{P}_{n-r}\,{P}_{n+r}=&(-1)^{n-r}\,{P}_{r}^2,\\
{Q}_{m}{Q}_{n}-{Q}_{m+r}{Q}_{n-r}=&(-1)^{n-r+1}{P}_{m+r-n}\,{P}_{r},\\
Q_{n}\,Q_{n}-Q_{n-r}\,Q_{n+r}=&(-1)^{n-r+1}\,P_{r}^2. 
\end{array}
\end{equation*}
\end{proof}

\end{document}